\documentclass[a4paper]{amsart}
\usepackage[T1]{fontenc}
\usepackage{lmodern}
\usepackage{amssymb}
\usepackage[all]{xy}
\usepackage[inline,shortlabels]{enumitem}
\usepackage{nicefrac,stmaryrd}
\usepackage{graphicx, import}
\usepackage{mathtools}
\usepackage[british]{babel}
\usepackage{microtype}

\usepackage{tikz-cd, tikz}
\usepackage[pdftitle={Haar system on fundamental groupoid},
pdfauthor={Holkar, Hossain},
  pdfsubject={} 
]{hyperref}
\usepackage[lite]{amsrefs}
 \usepackage{amsthm,amsfonts}

 \usepackage{caption}

\numberwithin{equation}{section}
\theoremstyle{plain}
\newtheorem{theorem}[equation]{Theorem}
\newtheorem*{theorem*}{Theorem}
\newtheorem{lemma}[equation]{Lemma}
\newtheorem{lemdef}[equation]{Lemma and definition}
\newtheorem{proposition}[equation]{Proposition}
\newtheorem{corollary}[equation]{Corollary}

\theoremstyle{definition}
\newtheorem{definition}[equation]{Definition}
\newtheorem{observation}[equation]{Observation}

\theoremstyle{remark}
\newtheorem{remark}[equation]{Remark}
\newtheorem{example}[equation]{Example}

\newcommand{\LTr}{\ell}
\newcommand*{\defeq}{\mathrel{\vcentcolon=}}
\newcommand{\Z}{\mathbb Z}
\newcommand{\C}{\mathbb C}
\newcommand*{\base}[1][G]{{#1}^{(0)}}
\newcommand{\nb}{\nobreakdash}
\newcommand{\7}{\backslash}
\newcommand{\inverse}{^{-1}}
\newcommand{\supp}{\textup{supp}}
\newcommand{\iso}{\simeq}
\newcommand{\FGd}{\mathrm{\Pi}_1}
\newcommand{\FGp}{\mathrm{\pi}_1}

\newcommand{\Contc}{\mathrm{C_c}}
\newcommand{\Contz}{\mathrm{C_0}}
\newcommand{\PathS}[1]{\mathrm{P}\!#1}

\newcommand*{\dd}{\mathrm d}
\newcommand*{\Cst}{\textup C^*}

\title[Haar systems on the fundamental groupoid]{Topological
  fundamental groupoid. III. Haar systems on the fundamental groupoid}

\author{Rohit Dilip Holkar} \email{rohit.d.holkar@gamil.com}

\author{Md Amir Hossain}
\email{mdamir18@iiserb.ac.in}

\address{Department
	of Mathematics, Indian Institute of Science Education and Research
	Bhopal, Bhopal Bypass Road, Bhauri, Bhopal 462 066, Madhya Pradesh,
	India.}

      \thanks{\emph{Subject class.} 22A22, 28C10, 28C15}
      
\begin{document}
\maketitle{}

\begin{abstract}
  Let $X$ be a path connected, locally path connected and semilocally
  simply connected space; let $\tilde{X}$ be its universal cover. We
  discuss the existence and description of a Haar system on the
  fundamental groupoid $\Pi_1(X)$ of $X$. The existence of a Haar
  system on $\Pi_1(X)$ is justified when $X$ is a second countable,
  locally compact and Hausdorff. We provide equivalent criteria for
  the existence of the Haar system on a locally compact (locally
  Hausdorff) fundamental groupoid in terms of certain measures on $X$
  and $\tilde{X}$.  $\mathrm{C}^*(\Pi_1(X))$ is described using a
  result of Muhly, Renault and Williams. Finally, two formulae for
  the Haar system on $\Pi_1(X)$ in terms of measures on $X$ or
  $\tilde{X}$ are given.
\end{abstract}

\tableofcontents{}

\section*{Introduction}
\label{sec:introduction}

This is the last instalment in the series of articles discussing the
fundamental groupoid. In the first
article,~\cite{Holkar-Hossain2023Top-FGd-I}, we equip the fundamental
groupoids of \emph{nice} spaces with a canonical topology. Then we
discuss what conditions on the nice space~\(X\) make its fundamental
groupoid~\(\FGd(X)\) locally compact or Hausdorff or second countable.
In the current article, we discuss the Haar system on the fundamental
groupoid.

In order to study the \(\Cst\)\nb-algebra of the fundamental groupoid,
a Haar system is a necessary ingredient. However, even experts may
notice the absence of any description of the Haar system on a locally
compact fundamental groupoid in the literature on locally compact
groupoid and \(\Cst\)\nb-algebras.  Nonetheless, in~\cite[page
19]{Muhly-Renault-Williams1987Gpd-equivalence}, Muhly, Renault and
Williams mention that the structure of the \(\Cst\)\nb-algebra of a
fundamental groupoid seems known to \emph{cognoscenti}. In fact, their
Theorem~3.1 in~\cite{Muhly-Renault-Williams1987Gpd-equivalence} gives
a description of the \(\Cst\)\nb-algebra of a fundamental
groupoid. One may guess that the experts might be alluding to
geometric situations---for example, the fundamental groupoid of
certain manifolds or Lie groups---wherein the Haar system show up
naturally.

In a relatively recent article, Williams' give results about the
existence of the Haar system on equivalent
groupoid~\cite{Williams2016Haar-systems-on-equivalent-gpds}. This
result and the fact that the fundamental group and groupoid of a space
are equivalent,~\cite{Muhly-Renault-Williams1987Gpd-equivalence},
implies that a locally compact, Hausdorff and second countable
fundamental groupoid carries a Haar system. We express this
observation as a formal result in
Corollary~\ref{cor:sec-cble-gpd-Haar-sym}.

Going beyond the second countable Hausdorff case, we discuss the
existence of Haar systems the fundamental groupoids that are locally
compact and locally Hausdorff; Corollary~\ref{cor:Haar-sys-general}
provides two criteria for the existence of this Haar system. These
criteria are given in terms of existence of certain measures on the
space~\(X\) itself or its universal covering space~\(\tilde{X}\). We
provide examples, Example~\ref{exa:Haar-system-general}, supporting
this general case of Corollary~\ref{cor:Haar-sys-general}.

In case of locally compact spaces, the first example is of a
topological group. Let \(X\) be a locally path connected, semilocally
simply connected and locally compact \emph{group}. We know that the
fundamental groupoid of~\(X\) a transformation groupoid for a group
action; this is known from the first article of this series or see
Example~\ref{exa:fgd-of-gp}. Now, \(\FGd(X)\) being a groupoid of
transformation for a group action, the existence of Haar system on it
(Theorem~\ref{thm:Haar-measure-Gpd-of-Gp}) can be easily seen.

After locally compact group, assume that~\(X\) is a locally compact
\emph{nice} space so that~\(\FGd(X)\) is topological.
Corollary~\ref{cor:Haar-sys-general} gives two equivalent conditions
for the existence of a Haar system on~\(\FGd(X)\);
Corollary~\ref{cor:Haar-sys-general} follows from
Theorem~\ref{thm:Haar-sys-general}. What is the purpose of this
theorem? As fundamental groupoids are prototypical examples of locally
trivial, we could apply the proofs written for fundamental groupoids
word-to-word for locally trivial
groupoids. Theorem~\ref{thm:Haar-sys-general} is a consequence of this
generalisation.

The next two heuristics lay behind
Corollary~\ref{cor:Haar-sys-general}: firstly that as the covering
spaces and the fundamental groupoid of~\(X\) can be intrinsically
built from~\(X\) itself (to be precise, using the paths in~\(X\)), we
expect that the Haar systems on~\(\FGd(X)\) should also be derivable
from the measures on~\(X\). Secondly, as the fundamental groupoid is
formed by \emph{pasting} the universal covering spaces (as fibres over
the range map), one may expect to express the Haar system
on~\(\FGd(X)\) in terms of the measures on the universal covering
space~\(\tilde{X}\).

The last heuristics also leads to formulae for the Haar system
on~\(\FGd(X)\) using the measures on~\(\tilde{X}\)
(Equation~\eqref{eq:measure-lambda}) and~\(X\)
(Equation~\eqref{eq:nu-to-haar-system}). Proof of
Equation~\eqref{eq:measure-lambda} is not straightforward as the idea
behind. We express~\(\FGd(X)\) as a certain quotient of product of the
universal covering spaces (\cite[Section
3.1]{Holkar-Hossain2023Top-FGd-I}) using the fact that it is
equivalence to the fundamental group. And the some standard technical
computation are involved.

Once we equip the fundamental groupoid with a Haar system, a result of
Muhly, Renault and
Williams~\cite[Theorem~3.1]{Muhly-Renault-Williams1987Gpd-equivalence}
describes~\(\Cst(\FGd(X))\), see Proposition~\ref{prop:cst-of-fgd}.

\medskip

\paragraph{\emph{Organisation of the article:}}

Section~\ref{sec:prilim} contains preliminaries: topological
groupoids, their actions, their equivalences and the Haar systems on
them. This section also contains basic definitions and collection of
required results about continuous families of measures.

Section~\ref{sec:haar-system} contains the main results. Its
Subsection~\ref{sec:case-second-count} discusses the Haar system on a
locally compact, Hausdorff and second countable locally trivial
groupoid. The second Subsection~\ref{sec:non-second-countable}
discusses the general case of locally compact locally trivial
groupoids. Finally, Section~\ref{sec:formulae-haar-system} contains
description of the Haar system on~\(\FGd(X)\) in terms of certain
measures on~\(X\) or the universal cover of~\(X\).

\section{Preliminaries}
\label{sec:prilim}

Most of the content of Sections~\ref{sec:topol-group}
and~\ref{sec:actions-groupoids} is \emph{copied} from the second
article of this series.

\subsection{Topological groupoids}
\label{sec:topol-group}

We continue to follow the conventions established in the earlier
articles~\cite{Holkar-Hossain2023Top-FGd-I}~\cite{Holkar-Hossain2023Top-FGd-II}.
Nonetheless, here is a quick recap of notation about groupoids: for
us, a \emph{groupoid}~\(G\) is a small category in which every arrow
is invertible. We abuse the notation and consider the set of
units~\(\base\) a subset of~\(G\). It is standard exercise that for an
element \(\gamma \in G\), \(\gamma \inverse \gamma = s(\gamma)\) and
\(\gamma \gamma \inverse = r(\gamma)\) are the range and source
of~\(\gamma\). The fibre product
\(G\times_{s, \base, r} G = \{(\gamma, \eta) \in G\times G : s(\gamma)
=r(\eta)\}\) is called the set of all composable pairs of \(G\) and
it's denoted by \(G^{(2)}\).

The groupoid~\(G\) is called \emph{topological} if it carries a
topology in which the source map~\(s\colon G\to \base\), range
map~\(r\colon G\to \base\), the inversion
map~\(\mathrm{inv}\colon G\to G\) and the
multiplication~\(m\colon G^{(2)} \to G\) are continuous; here the
space of units is given the subspace topology, and
\(G^{(2)}\subseteq G\times G\) carries the subspace topology.

The topological groupoid~\(G\) is called locally compact (or Hausdorff
or second countable) if the topology is locally compact (respectively,
Hausdorff or second countable) plus the space of units is
Hausdorff. For us locally compact spaces are not necessarily
Hausdorff, see~\cite{Holkar-Hossain2023Top-FGd-I}. However, unless the
reader is bothered about non-Hausdorff case, they may simply consider
locally compact as locally compact Hausdorff. Second countable or
paracompact spaces are assumed to be Hausdorff. We refer the reader to
Renault's book~\cite{Renault1980Gpd-Cst-Alg} and Tu's
article~\cite{Tu2004NonHausdorff-gpd-proper-actions-and-K} for basics
of locally compact groupoids and their actions.

Given units \(x, y\in \base\), we define the following closed subspace
of~\(G\):
\[
  G^x \defeq r^{-1}(x),\quad G_y \defeq s^{-1}(y)\quad \text{ and }
  \quad G^x_y \defeq G^x\cap G_y.
\]
In fact, \(G^x_x\) is a topological group called the \emph{isotropy
  group} at~\(x\). In general, for sets~\(A,B\subseteq \base\), we
define
\[
  G^A \defeq r^{-1}(A),\quad G_B \defeq s^{-1}(B)\quad \text{ and }
  \quad G^A_B \defeq G^A\cap G_B.
\]

Following are some examples of groupoids we shall require later.

\begin{example}[Groupoid of trivial equivalence relation]
  \label{exa:gpd-triv-equi}
  Let \(X\) be a space and \(X\times X\) the trivial equivalence
  relation equipped with the product topology. Then \(X\times X\) is
  also a groupoid: the space of units is the diagonal
  \(\{(x,x) : x \in X\}\subseteq X\times X\) which we identify
  with~\(X\). The range and source maps are the projections onto first
  and second factors, respectively. The composite of arrows
  \((x,y)(y,z) = (x,z)\) for \(x,y,z\in X\).
\end{example}

Next we discuss the fundamental groupoid. Consider a space~\(X\). For
a set~\(U\subseteq X\), we write~\(\PathS{U}\) for the set of all
paths in~\(U\). For a path~\(\gamma \in \PathS{X}\), \(\gamma(0)\) is
called the initial or starting point and~\(\gamma(1)\) the terminal or
end point of~\(\gamma\). For~\(\gamma\) as before, \(\gamma^-\)
denotes the path \emph{opposite} to~\(\gamma\). The concatenation of
paths is denoted by~\(\oblong\). The fundamental groupoid of~\(X\) is
denoted by~\(\FGd(X)\). The fundamental group of~\(X\) at~\(x\in X\)
is denoted by~\(\FGp(X,x)\). We shall use the fact that~\(\FGd(X)\) is
the quotient of~\(\PathS{X}\) by the equivalence relation of endpoint
fixing path homotopy.  If \(X\) path connected, we simply
write~\(\FGp(X)\) instead of~\(\FGp(X,x)\).  Finally, since we direct
the arrows in a groupoid from \emph{right to left}, we shall think
that a path starts from right and ends on left in oppose to the
standard convention.
      
      \begin{example}[The fundamental groupoid]\label{exa:fund-gpd}
        Let \(X\) be a locally path connected and semilocally simply
        connected space. Then, we prove in the preceding
        article~\cite{Holkar-Hossain2023Top-FGd-I}, that the
        fundamental groupoid~\(\FGd(X)\) of~\(X\) can be equipped with
        a topology so that it becomes a topological groupoid. Equip
        the set of all paths, \(\PathS{X}\), in~\(X\) with the
        compact-open topology. Then the quotient topology
        on~\(\FGd(X)\) induced by the compact-open topology make the
        fundamental groupoid topological~\cite[Theorem
        2.8]{Holkar-Hossain2023Top-FGd-I}. We call this quotient
        topology the \emph{CO' topology}.

        There is another natural way to topologise~\(\FGd(X)\) as
        follows. For path connected and relatively inessential open
        sets~\(U,V\in X\) and a path~\(\gamma\) in~\(X\) starting at a
        point in~\(V\) and ending at a point in~\(U\), define for the
        following subset of~\(\FGd(X)\)
        \[
          N([\gamma], U,V) \defeq \{[\delta \oblong \gamma \oblong
          \omega] : \delta\in \PathS{U} \text{ with } \delta(0) =
          \gamma(1), \text{ and } \omega\in \PathS{V} \text{ with
          }\omega(1) = \gamma(0)\}.
        \]
        Then the sets of above form a basis for a topology
        on~\(\FGd(X)\) which we call \emph{the UC
          topology}. Proposition~2.4
        in~\cite{Holkar-Hossain2023Top-FGd-I} shows that for a locally
        path connected and semilocally simply connected space~\(X\),
        the UC and CO' topologies on the fundamental groupoid are
        same.

        In the groupoid~\(\FGd(X)\), the range and source maps (which
        are basically the evaluations at~\(1\) and~\(0\),
        respectively) are open~\cite[Corollay
        2.7]{Holkar-Hossain2023Top-FGd-I}. The space of
        units~\(\base[\FGd(X)]\) consists of constant paths and can be
        identified with~\(X\)~\cite[Corollary
        2.9(1)]{Holkar-Hossain2023Top-FGd-I}. Assume that~\(X\) is
        also path connected. Then, for a unit~\(x\in X\), the
        fibres~\(\FGd(X)_x\) or~\(\FGd(X)^x\) can be identified with
        the simply connected covering space of~\(X\) (constructed
        using either the paths starting at~\(x\) or ending
        at~\(x\))~\cite[Corollary~2.9(2)]{Holkar-Hossain2023Top-FGd-I}. Moreover,
        the isotropy group \(\FGd(X)_x^x\) at~\(x\) is basically the
        fundamental group of~\(X\), and it is discrete
        ~\cite[Corollary~2.9(3)]{Holkar-Hossain2023Top-FGd-I}.

        The fundamental groupoid is Hausdorff (or locally compact or
        second countable) \emph{iff} the underlying spaces is
        so~\cite[Section 3]{Holkar-Hossain2023Top-FGd-I}.
      \end{example}

      \begin{example}[Fundamental groupoid of a group]
        \label{exa:fgd-of-gp}
        Assume that~\(X\) is a path connected, locally path connected
        and semilocally simply connected topological
        \emph{group}. Let~\(H\) be its covering group with the group
        homomorphism~\(p\colon H\to X\) as the covering
        map. Then~\(H\) acts on~\(X\) through~\(p\). Theorem~2.21
        in~\cite{Holkar-Hossain2023Top-FGd-I} prove that the
        topological fundamental groupoid~\(\FGd(X)\) is isomorphic to
        the transformation groupoid~\(H \ltimes X\) of above action
        of~\(H\) on~\(X\).
      \end{example}

      \subsection{Locally trivial groupoids}
      \label{sec:locally-triv-group}

      A topological groupoid~\(G\) is called \emph{locally trivial} if
      for given~\(x\in \base[G]\), the restriction of source
      map~\(s|_{G^x}\colon G^x\to \base[G]\) is a (surjective) local
      homeomorphism. As the inversion is a homeomorphism of~\(G\), we
      may equivalently demand that~\(r|_{G_x}\) is a local
      homeomorphism.

      The subspaces \(G_x\) or \(G^x\) of \(G\) are also called
      \emph{transversals},
      see~\cite{Muhly-Renault-Williams1987Gpd-equivalence}.

      \begin{observation}
        \label{obs:discrete-isotropy-loc-triv-gpd}
        Note that since \(s|_{G_x}\) is a local homeomorphism for a
        locally trivial groupoid~\(G\), the isotropy
        \(G_x^x = s|_{G_x}\inverse(\{x\})\) is discrete.
      \end{observation}

      A groupoid~\(G\) is called \emph{transitive} if
      given~\(x,y\in\base[G]\), there is an arrow~\(\gamma\in G\) with
      \(x=r(\gamma)\) and \(y=s(\gamma)\). Clearly, a locally trivial
      groupoid is transitive.

      \begin{example}[See~\cite{Holkar-Hossain2023Top-FGd-I}]
        \label{exa:fgd-loc-triv}
        Let \(X\) be a path connected, locally path connected and
        semilocally simply connected space. Then its fundamental
        groupoid~\(\FGd(X)\) is locally trivial. For \(x\in X\),
        \(\FGd(X)^x\) is homeomorphic to the simply connected covering
        space of~\(X\) and \(s|_{\FGd(X)^x}\colon \tilde{X} \to X\) is
        basically the covering map. The left multiplication action of
        the isotropy~\(\FGp(X,x)\) on~\(\FGd(X)^x\) is the action by
        deck transformation.
      \end{example}

      \subsection{Actions of groupoids}
      \label{sec:actions-groupoids}

\begin{definition}
  Let \(G\) be a locally compact, Hausdorff groupoid, and let \(X\) be
  a topological space with a continuous momentum map
  \(r_X\colon X \to \base\). We call \(X\) is a left \(G\)\nb-space
  (or \(G\) act on \(X\) from left) if there is a continuous map
  \( \sigma \colon G \times_{s, \base, r_X} X \to X\) satisfying the
  following conditions:
  \begin{enumerate}
  \item \( \sigma (r_X(x), x) = x\) for all \(x\in X\);
  \item if \((\gamma, \eta) \in G^{(2)}\) and
    \((\eta, x) \in G \times_{s, \base, r_X} X \), then
    \((\gamma\eta, x), (\gamma, \sigma(\eta, x) ) \in G \times_{s,
      \base, r_X} X \) and
    \(\sigma ((\gamma \eta) , x) = \sigma (\gamma ,\sigma(\eta, x))\).
  \end{enumerate}
\end{definition}
      
We shall abuse the notation \(\sigma (\gamma, x)\) and simply write
\(\gamma \cdot x\) or \(\gamma x\). We shall often say `\(X\) is a
left (or right) \(G\)\nb-space' for a groupoid~\(G\); here it will be
\emph{tacitly assumed} that~\(r_X\) (respectively, left) is the
momentum map. The range (or source) map is the momentum map for the
left (respectively, right) multiplication action of a groupoid on
itself.

A groupoid~\(G\) acts on its space of units, from left, as follows:
for \(\gamma\in G\) and \(x\in G\), the action is defined if
\(s(\gamma) = x\) and is given by~\(\gamma x = r(\gamma)\). The
identity map on \(\base\) is the momentum map for this
action. Similarly a right action of~\(G\) on~\(\base\) is defined.

Given \(G\)\nb-spaces \(X\) and \(Y\), by an \emph{equivariant map} we
mean a function \(f\colon X\to Y\) such that \(r_Y\circ f = r_X\) and
\(f(\gamma x) = \gamma f(x)\) for all composable pairs
\((\gamma,x)\in G\times_{s, \base, r_X} X\).

\begin{example}[Transformation groupoid]
  \label{exa:transormation-gpd}
  For a continuous (right) action of a groupoid~\(G\) on a
  space~\(X\), one can construct the transformation groupoid which is
  denoted by~\(X\rtimes G\). The underlying space of the groupoid is
  the fibre product~\(X\times_{s_X, \base[G], r} G\); two elements
  \((x,g)\) and \((y,t)\) in~\(X\rtimes G\) are composable iff
  \(y= x\cdot g~\&~(g,t) \in G^{(2)}\), and the composition is given
  by \((x,g)(y,t) \defeq (x, gt)\); the inverse of~\((x,g)\) is given
  by \((x,g)^{-1} \defeq (x\cdot g, g^{-1})\).  For a left
  \(G\)\nb-space \(Y\), the transformation groupoid is defined
  similarly and is denoted by~\(G\ltimes Y\).
\end{example}

Assume that \(X\) is a \(G\)\nb-space for a groupoid~\(G\). While
studying the groupoid actions, the following map
\begin{equation}\label{eq:kine}
  a\colon G\times_{s, \base[G], r_X} X \to X\times X, \quad a\colon
  (\gamma, x)  \mapsto (\gamma x, x)
\end{equation}
turns out useful. Although, this map does not have a standard name,
for the current article we call it \emph{the kinetics}\footnote{A
  better name is welcome!} or \emph{the map of the kinetics} of the
action. Observation~2.10 in~\cite{Holkar-Hossain2023Top-FGd-I} shows
that the map of kinetics for the action of~\(\FGd(X)\) on its space of
units is a local homeomorphism where~\(X\) is a locally path connected
and semilocally simply connected space.

Recall from~\cite[Section 1.2]{Holkar-Hossain2023Top-FGd-I}, that an
action of~\(G\) on~\(X\) is \emph{free} iff the map of kinetics of the
action is one-to-one. And the action is \emph{proper} iff the map of
the kinetics is a proper map.

A groupoid~\(G\) is called proper if its action on the space of units
is proper.

\subsection{Equivalence of groupoids}
\label{sec:equiv-group}

\begin{definition}[\cite{Muhly-Renault-Williams1987Gpd-equivalence}]
  \label{def:gpd-equi}
  Let \(G\) and \(H\) be locally compact groupoids. An equivalence
  between~\(G\) and~\(H\) is a space~\(X\) carrying a left
  \(G\)\nb-action, right~\(H\)\nb-action and following holds:
  \begin{enumerate}
  \item the actions commute, that is,
    \((\gamma x)\eta = \gamma (x\eta)\) for appropriate
    \(\gamma \in G, x\in X\) and~\(\eta\in H\).
  \item The action of \(G\) on~\(X\) is free and proper.
  \item The action of \(H\) on~\(X\) is free and proper.
  \item The momentum map~\(s_X\) induces a
    homeomorphism~\(G\7 X \to \base[H]\).
  \item The momentum map~\(r_X\) induces a
    homeomorphism~\(X/H \to \base[G]\).
  \end{enumerate}
\end{definition}

\begin{remark}
  Note that if \(X\) implements an equivalence between \(G\) and
  \(H\), then the restriction of the map of kinetics
  \[
    k\colon G\times_{s, \base[G], r_X} X \to X\times_{s_X, \base[H],
      s_X} X
  \]
  obtained by restricting the range is a
  homeomorphism. Definition~\ref{def:gpd-equi}(2) implies that this
  map is a continuous bijection onto a closed subspace
  of~\(X\times X\); the fourth condition in
  Definition~\ref{def:gpd-equi} implies that the fibre product
  \(X\times_{s_X, \base[H], s_X} X\) is the range of this map;
  finally, since~\(k\) is proper, it is also closed.
\end{remark}

We shall need the following well-known fact:

\begin{lemma}[{\cite[Lemma
    2.38]{Williams2019A-Toolkit-Gpd-algebra}}] \label{lem:equi-gpd-construction}
  Let \(G\) and \(H\) be locally compact groupoids and~\(X\) be an
  equivalence between them. Then~\(H\) is isomorphic to the
  topological groupoid~\(G\7 (X\times_{r_X, \base[G], r_X} X)\).
\end{lemma}

We briefly recall the quotient space and its groupoids structure
from~\cite{Muhly-Renault-Williams1987Gpd-equivalence}: the considered
action of~\(G\) on~\((X\times_{r_X, \base[G], r_X} X)\) is the
diagonal action, that is, for \(\gamma\in G\) and
appropriate~\((x,y)\in (X\times_{r_X, \base[G], r_X} X)\),
\(\gamma(x,y) = (\gamma x, \gamma y)\).  Let \([(x,y)]\) denote the
equivalence class of \((x,y)\in (X\times_{r_X, \base[G], r_X} X)\)
in~\(G\7(X\times_{r_X, \base[G], r_X} X)\). Denote the space of units
of \(G\7(X\times_{r_X, \base[G], r_X} X)\) is identified
with~\(G\7 X\). The range, source and inverse of
\([(x,y)]\in G\7 (X\times_{r_X, \base[G], r_X} X)\) are
respectively~\([x]\in G\7 X\) and~\([y]\in G\7X\) and \([(y,x)]\). Two
arrows \([(x,y)], [(w,z)]\in G\7 (X\times_{r_X, \base[G], r_X} X)\)
are composable \emph{iff} \(y=w\) and the
composition~\([(x,y)][(y,z)] = [(x,z)]\). All these operations are
well-defined.

Let \([(x,y)]\in G\7 (X\times_{r_X, \base[G], r_X} X)\). Then
Definition~\ref{def:gpd-equi}(4) assure existence of a unique arrow
\(\eta_{[(x,y)]} \in H\) such that \(x \eta_{[(x,y)]} = y\). The
assignment \([(x,y)] \mapsto \eta_{[(x,y)]}\) defines the isomorphism
\(G\7 (X\times_{r_X, \base[G], r_X} X) \to H\).

Above lemma is stated in~\cite[Lemma
2.38]{Williams2019A-Toolkit-Gpd-algebra} for \emph{Hausdorff
  groupoids}, however, one can check that the proof works for locally
Hausdorff, locally compact groupoids in general. Alternatively, the
lemma can be derived directly from~\cite[Lemma
4.5(d)]{Kumjian-Muhly-Renault1998Brauer-gp-of-lc-gpd}.

We need last lemma particularly in the following form: let
\(x\in\base[G]\) in a locally compact transitive
groupoid~\(G\). Then~\(G^x\) implements an equivalence of groupoids
between the isotropy~\(G_x^x\) and~\(G\)
itself~\cite{Muhly-Renault-Williams1987Gpd-equivalence}. As a
consequence~\(G\) can be identified with the
quotient~\(G^x_x\7(G^x\times G^x)\). In particular, this isomorphism
holds if~\(G\) is locally trivial.

\subsection{Families of measures}
\label{sec:measure-families}

By a measure on a locally compact space, we mean a positive Radon
measure. Such measures are exactly the positive linear
functional~\(\Contc(X)\to \C\) due to the Riesz
theorem~\cite{Rudin1987Real-And-Complex}. Therefore, if~\(\mu\) is
such a measure on a locally compact space~\(X\), then we also
write~\(\mu(f)\) for~\(\int_X f\,\dd\mu\). The support of~\(\mu\) is
smallest closed set~\(S\subseteq X\)
with~\(\mu(X\setminus S)=\emptyset\). We shall work with fully
supported measures, that is, the measures whose support is the whole
space. Here~\(\Contc(X)\) denotes the normed involutive algebra of
compactly supported continuous complex valued function on~\(X\)
equipped with the uniform norm.

Consider a homeomorphism \(\phi\colon X \to Y\) of locally compact
spaces. A measure~\(\mu\) on~\(Y\) induces the \emph{pullback}
measure~\(\phi^*\mu\) on~\(X\) which is defined by
\(\phi^*\mu(f) = \mu(f\circ \phi\inverse)\) for \(f\in \Contc(X)\).

Let \(X\) and~\(Y\) be locally, compact Hausdorff spaces and
\(\pi\colon X\to Y\) a continuous open surjection. By a
\emph{continuous family of measures along~\(\pi\)} we mean a family of
measures~\(\mu\defeq \{\mu_y\}_{y \in Y}\) such that for
each~\(y\in Y\), \(\mu_y\) is a measure on~\(X\) with supported
in~\(\pi\inverse(y)\), and for \(f\in \Contc(X)\), the following
complex valued function on~\(Y\) given by \( y \mapsto \mu_y(f)\), for
\(y\in Y\), is in~\(\Contc(Y)\).

We denote the above function \(y \mapsto \mu_y(f)\)
by~\(\mu(f)\). Thus \(\mu\colon \Contc(X) \to \Contc(Y)\) is a map of
complex vector spaces; \cite[Lemme
1.1]{Renault1985Representations-of-crossed-product-of-gpd-Cst-Alg}
says that~\(\mu\) is surjective.

If \(\nu\) is a measure on~\(Y\), then \(\nu\circ \mu\) is a measure
on~\(X\); for \(f\in \Contc(X)\), \(\nu\circ \mu(f) = \nu(\mu(f))\).

We frequently refer a continuous family of measures as a family of
measures.

\begin{example}
  \label{exa:local-homeo-family-of-measures}
  Let \(\pi\colon X\to Y\) be a surjective local
  homeomorphism. For~\(y\in Y\), define \(\mu_y\) to be the counting
  measure on~\(\pi\inverse(y)\). Then \(\mu\defeq \{\mu_y\}_{y\in Y}\)
  is a continuous family of measures along~\(\pi\). For
  \(f\in \Contc(X)\) and \(y\in Y\),
  \[
    \mu_y(f) = \sum_{x\in \pi\inverse(y)} f(x).
  \]
  In particular, if~\(f\) is supported in an open set~\(U\) such that
  \(\pi|_U\) is a homeomorphism, then
  \[
    \mu_y(f) =
    \begin{cases}
      f((\pi|_U)^{-1}(y)) & \text{ if } y\in \pi(U),\\
      0 & \text{ otherwise.}
    \end{cases}
  \]
\end{example}

\begin{example}
  \label{exa:fam-measures-along-quotient}
  Let \(X\) be a locally compact proper left \(G\)\nb-space for a
  locally compact groupoid~\(G\). Let \([x]\in G\7 X\) denote the
  equivalence class of~\(x\in X\). Then a Haar system~\(\alpha\)
  on~\(G\) induces a (continuous) family of
  measures~\(\tilde{\alpha}=\{\tilde{\alpha}_{[x]}\}_{[x]\in G\7 X}\)
  along the quotient map~\(q\colon X\to G\7 X\). This can be seen as a
  consequence of Lemma~1.3(i)
  in~\cite{Renault1985Representations-of-crossed-product-of-gpd-Cst-Alg}. This
  family of measures is essentially \emph{averaging} along the fibres:
  for \(g\in \Contc(X)\) and \([x]\in G\7 X\),
  \[
    \int_X g\, \dd\tilde{\alpha}_{[x]} \defeq \int_G g(\gamma\inverse
    x) \, \dd\alpha^{r_X(x)}(\gamma).
  \]
  The last integrand is well-defined (due to the properness of the
  \(G\)\nb-action) and does not depend on the choice of~\(y\in [x]\).
\end{example}

Assume that \(X\) and~\(Y\) above are \(G\)\nb-spaces for a
groupoid~\(G\) and \(\pi\colon X\to Y\) is a \(G\)\nb-equivariant open
mapping. Suppose a continuous family of
measures~\(\mu\defeq \{\mu_y\}_{y\in Y}\) along~\(\pi\) is given. We
call~\(\mu\)\emph{ \(G\)\nb-equivariant} if for each composable
pair~\((\gamma, y)\in G\times Y\),
\[
  \int_{\pi\inverse(\gamma y)} f(x) \, \dd\mu_{\gamma y}(x) =
  \int_{\pi\inverse(y)} f(\gamma\inverse x) \, \dd\mu_{y}(x)
\]
for \(f\in \Contc(X)\).

Note that if~\(G\) is a group, then a \(G\)\nb-equivariant family for
the constant map \(X\to \{*\}\) is essentially a \(G\)\nb-invariant
measure.

A \emph{Haar system} on a locally compact groupoid~\(G\) is a
continuous family of fully supported
measures~\(\alpha\defeq \{\alpha^u\}_{u\in \base}\) along the range
map~\(r\colon G\to \base\) that is also \(G\)\nb-equivariant of the
left multiplication action of~\(G\) on itself and the standard action
of~\(G\) on~\(\base\). Thus, the measure~\(\alpha^u\) is supported
on~\(G^u\) and for \(f\in \Contc(G)\) and \(\gamma\in G\),
\[
  \int_{G^{s(\gamma)}} f(\eta)\, \dd\alpha^{s(\gamma)}(\eta) =
  \int_{G^{r(\gamma)}} f(\gamma \eta)\, \dd\alpha^{r(\gamma)}(\eta).
\]

It is well-known that, unlike groups, a locally compact groupoid may
not have a Haar system.

\begin{example}\label{exa:haar-sys}
  \begin{enumerate}
  \item The Haar measure on a locally compact group is a Haar system.
  \item The transformation groupoid of group action, \(X \rtimes G\),
    in Example~\ref{exa:transormation-gpd} carries a Haar system. Let
    \(\lambda\) be the Haar measure on~\(G\); and for \(x\in X\), let
    \(\delta_x\) be the point mass at~\(x\) in~\(X\). For each
    \(x\in X\), define \(\alpha_x=\delta_x \times \lambda\). Then
    \(\{\alpha_x\}_{x\in X}\) is a Haar system on~\(X\rtimes G\).
  \item Let \(X\) be a locally compact space and \(X\times X\) be the
    groupoid of trivial equivalence on~\(X\), see
    Example~\ref{exa:gpd-triv-equi}. Assume that \(\lambda\) is fully
    supported measure on~\(X\). Then
    \(\{\delta_x\times\lambda\}_{x\in X}\) is a Haar system
    on~\(X\times X\).
  \end{enumerate}
\end{example}

Following are two results about Haar system that we shall need.

\begin{lemma}[Theorem 2.1~\cite{Williams2016Haar-systems-on-equivalent-gpds}]\label{lem:Williams-result}
  Let~\(G\) and~\(H\) be locally compact, Hausdorff and second
  countable \emph{equivalent} groupoids. If~\(H\) has a Haar system,
  then so does~\(G\).
\end{lemma}

\begin{lemma}[Proposition 5.2
  \cite{Kumjian-Muhly-Renault1998Brauer-gp-of-lc-gpd}]
  \label{lem:Brauer-gp-result-Haar-system}
  Let~\(G\) and~\(H\) be locally compact (and not necessarily
  Hausdorff or second countable) \emph{equivalent} groupoids. Suppose
  a locally compact \(G\)-\(H\)\nb-bispace \(X\) implements this
  equivalence. Then~\(G\) has a Haar system \emph{iff} the momentum
  map the the \(H\)\nb-action \(s_X \colon X \to \base[H]\) has an
  \(H\)\nb-invariant fully supported family measures.
\end{lemma}

We end the preliminaries by recalling the following lemma from
Bourbaki:

\begin{lemdef}[\cite{Bourbaki2004Integration-II-EN}, Lemma~1,
  Appendix~I]
  \label{lem:exist-cut-off-function}
  Let \(X\) be a locally compact Hausdorff space, \(R\) an open
  equivalence relation on \(X\), such that the quotient space \(X/R\)
  is paracompact. Let \(q \colon X \to X/R\) be the canonical quotient
  map. Then there is a real-valued continuous function \(e \geq 0\) on
  \(X\) such that:
  \begin{enumerate}
  \item \label{cond:cut-off-1} \(e\) is not identically zero on any
    equivalence for \(R\);
  \item \label{cond:cut-off-2} for every compact subset \(K\) of
    \(X/R\), the intersection \(q^{-1}(K) \cap \supp(e)\) is a compact
    set.
  \end{enumerate}
  We call the function~\(e\) a cutoff function for~\(q\).
\end{lemdef}

In particular, if \(p\colon \tilde{X} \to X\) is a covering map over a
locally path connected, semilocally simply connected and paracompact
space~\(X\), the~\(p\) has an associated cutoff function.

\section{A Haar system on a fundamental groupoid}
\label{sec:haar-system}

Although we aim to study the Haar system on fundamental groupoids, the
results in this section are worked out for locally trivial
groupoids. Fundamental groupoid being a locally trivial one, we do get
desired results. However, all the proofs go verbatim if one replaces a
locally trivial groupoid by~\(\FGd(X)\), and do the corresponding
identifications of subspaces and
maps. Corollaries~\ref{cor:sec-cble-gpd-Haar-sym},~\ref{cor:Haar-sys-general}
and Theorem~\ref{thm:Haar-measure-Gpd-of-Gp} are the key
results. Formulae in Section~\ref{sec:formulae-haar-system} shall
prove useful in computation.

\subsection{Case of second countable locally trivial groupoids}
\label{sec:case-second-count}

\begin{theorem}
  \label{thm:sec-cble-gpd-Haar-sym}
  Let \(G\) be a locally compact, Hausdorff, second countable and
  locally trivial groupoid. Then~\(G\) has a Haar system.
\end{theorem}

\begin{proof}
  Fix \(x\in \base[G]\). Since~\(G\) is second countable, the
  isotropy~\(G_x^x\) is a countable discrete group; the counting
  measures is the Haar measures on~\(G_x^x\). As \(G^x\) establishes
  an equivalence between~\(G_x^x\) and~\(G\), current theorem follows
  from Williams' Theorem~\ref{lem:Williams-result}.
\end{proof}

\begin{corollary}
  \label{cor:sec-cble-gpd-Haar-sym}
  Let \(X\) be a path connected, locally path connected, semilocally
  simply connected, locally compact, Hausdorff and second countable
  space. The fundamental groupoid~\(\FGd(X)\) has a Haar system.
\end{corollary}
\begin{proof}
  Since~\(X\) is locally compact, Hausdorff and second countable, so
  is~\(\FGd(X)\), see~\cite{Holkar-Hossain2023Top-FGd-I}. Now the
  result follows from last theorem.
\end{proof}

\subsection{The general case}
\label{sec:non-second-countable}

\begin{theorem}\label{thm:Haar-measure-Gpd-of-Gp}
  Let \(X\) be a path connected, locally path connected, semilocally
  simply connected and locally compact topological group. Then its
  fundamental groupoid~\(\FGd(X)\) has a Haar system.
\end{theorem}

\begin{proof}
  By Example~\ref{exa:fgd-of-gp} we know that~\(\FGd(X)\) is a
  transformation groupoid~\(H\ltimes X\); here~\(H\) is the universal
  covering space of~\(X\) which is also a locally compact
  group. Therefore, the Haar measure on~\(H\) induces a Haar system
  on~\(\FGd(X)\), see~Example~\ref{exa:haar-sys}(2).
\end{proof}

\begin{theorem}\label{thm:Haar-sys-general}
  Let \(G\) be a locally compact, locally trivial groupoid, and let
  \(x\in \base[G]\). Consider the following statements:
  \begin{enumerate}
  \item \(G\) has a Haar system;
  \item the transversal \(G^x\) has a \(G_x^x\)\nb-invariant fully
    supported Radon measure for the left multiplication action of
    \(G_x^x\) on~\(G^x\);
  \item \(\base[G]\) has a fully supported Radon measure.
  \end{enumerate}
  Then (1)\(\iff\)(2) and (3)\(\implies\)(2). If \(\base[G]\) is
  paracompact, then (2)\(\implies\)(3).
\end{theorem}

Note that a \(G_x^x\)\nb-invariant measure~\(\lambda\) on \(G^x\) is
fully supported \emph{iff} \(\lambda\) is nonzero.

\begin{proof}[Proof of Theorem~\ref{thm:Haar-sys-general}]
  (1)\(\iff\)(2): The transversal~\(G^x\) implements equivalence
  between~\(G^x_x\) and~\(G\). Therefore, the claim follows from
  Lemma~\ref{lem:Brauer-gp-result-Haar-system}.

  \noindent (3)\(\implies\)(2): Assume that~\(\nu\) is a fully
  supported measure on~\(\base[G]\).
  Let~\(\mu\defeq \{\mu_x\}_{x\in \base[G]}\) be the family of
  counting measures along the local
  homeomorphism~\(s|_{G^x}\colon G^x \to \base[G]\). Then the measure
  \(\nu \circ \mu\) is a \(G_x^x\)\nb-invariant fully supported
  measure on~\(G^x\). Since~\(\mu\) and~\(\nu\) have full support, so
  does~\(\nu\circ \mu\). Also note that~\(\mu\) is invariant for the
  left multiplication action of~\(G_x^x\) on~\(G^x\). Now
  \(G_x^x\)\nb-invariance of~\(\nu\circ\mu\) follows from a standard
  computation, for example, see~Proposition~3.1(i)
  in~\cite{Holkar2017Composition-of-Corr}.

  \noindent (2)\(\implies\)(3): Assume that~\(\base[G]\) is
  paracompact, and let~\(\lambda\) be a \(G_x^x\)\nb-invariant measure
  on~\(G^x\) having full support.

  Let \(e\) be the cutoff function for the local
  homeomorphism~\(s|_{G^x}\colon G^x \to \base[G]\). Getting the
  measure~\(\nu\) on~\(\base[G]\) from~\(\lambda\) is a standard
  process, for example, see~Proposition~3.2(ii)
  in~\cite{Holkar2017Composition-of-Corr}. We describe this
  construction next.

  Since \(s|_{G^x}\inverse(z) \cap \supp(e)\) is a finite set for
  \(z\in \base[G]\), we define \(h\colon G^x \to [0,\infty)\) as
  \[
    h(\tilde{x}) = \frac{e(\tilde{x})}{\sum_{\tilde{y}\in
        (s|_{G^x})\inverse(s|_{G^x}(\tilde{x}))} e(\tilde{y})}
  \]
  for \(\tilde{x}\in G^x\).  Then \(h\) is also a cutoff
  function. Additionally, for \(\tilde{z}\in G^x\) and
  \(x\in \base[G]\),
  \[
    0\leq h(\tilde{z}) \leq 1 \quad \text{ and } \quad
    \sum_{\tilde{x}\in s|_{G^x}\inverse(x)} h(\tilde{x}) = 1.
  \]

  Now for \(f\in \Contc(\base)\), define the positive Radon
  measures~\(\nu\) as follows:
  \[
    \nu(f) \defeq \int_{G^x} f\circ s|_{G^x}(\tilde{x}) \cdot
    h(\tilde{x}) \, \dd\lambda(\tilde{x}).
  \]
  The fact that~\(\nu\) is a \(G_x^x\)\nb-invariant measure on~\(G^x\)
  is a special case
  of~\cite[Proposition~3.2(ii)]{Holkar2017Composition-of-Corr}. The
  measure~\(\nu\) has full support because \(\lambda\) has full
  support and the cutoff function~\(h\) does not identically vanish on
  any fibre \(s|_{G^x}\inverse(z)\) for \(z\in \base[G]\).

  In fact, \cite[Proposition~3.2(iii)]{Holkar2017Composition-of-Corr}
  says that, for given~\(\lambda\), this~\(\nu\) is the unique measure
  with the property that~\(\lambda = \nu\circ \mu\).
\end{proof}

\begin{corollary}[Corollary of Theorem~\ref{thm:Haar-sys-general}]\label{cor:Haar-sys-general}
  Consider the following statements for a locally compact, path
  connected, locally path connected and semilocally simply connected
  topological space~\(X\):
  \begin{enumerate}
  \item the fundamental groupoid \(\FGd(X)\) has a Haar system;
  \item the universal covering space \(\tilde{X}\) has a
    \(\FGp(X)\)\nb-invariant fully supported Radon measure;
  \item \(X\) has a fully supported Radon measure.
  \end{enumerate}
  Then (1)\(\iff\)(2) and (3)\(\implies\)(2). If \(X\) is paracompact,
  then (2)\(\implies\)(3).
\end{corollary}

We derive the following standard fact from last corollary:

\begin{corollary}
  Let \(X\) be a locally compact, path connected, locally path
  connected and semilocally simply connected topological group. Then
  the Haar measure on its simply connected cover~\(\tilde{X}\) is
  \(\FGp(X)\)\nb-invariant.
\end{corollary}
\begin{proof}
  Consequence of Theorem~\ref{thm:Haar-measure-Gpd-of-Gp} and
  Theorem~\ref{thm:Haar-sys-general}(1)-(2).
\end{proof}

\begin{example} \label{exa:Haar-system-general} Consider a locally
  compact, path connected, locally path connected and semilocally
  simply connected space~\(X\).  Theorem~\ref{thm:Haar-sys-general}
  implies that in the following cases~\(\FGd(X)\) has Haar system.
  \begin{enumerate}
  \item If \(X\) is a homogeneous \(G\)\nb-space for some locally
    compact group~\(G\). In this case, it is well-known to experts
    that \(X\) carries a quasi-invariant measure
    (see~\cite{Folland1995Harmonic-analysis-book}) which has full
    support.
  \item Let \(X\) be a compact Hausdorff space and \(\sigma\) a
    homeomorphism of it. Then \(\Z\) acts on~\(X\) via this
    homeomorphism which gives the classical dynamical system
    \((X, \sigma)\). Assume that this action is transitive. Then~\(X\)
    has a \(\sigma\)\nb-invariant measure~\(\mu\) (\cite[Theorem
    VII.3.1]{Davidson1996Cst-by-Examples}). Due to the transitivity of
    the action, this is a measure on \(X\) with full support.
  \end{enumerate}
\end{example}

\subsection{Formulae for the Haar system}
\label{sec:formulae-haar-system}

Theorem~\ref{thm:Haar-sys-general} provides conditions of the
existence of the Haar system on~\(\FGd(X)\).  We can also describe the
Haar system concretely for a given measure on~\(\tilde{X}\) as in
Theorem~\ref{thm:Haar-sys-general}.  \smallskip

\subsubsection{Haar system in terms of measure on the transversal:}

Consider \(G\) as in Theorem~\ref{thm:Haar-sys-general}. Fix
\(x\in \base[G]\). Assume that a \(G_x^x\)\nb-invariant nonzero
measure \(\lambda\) on~\(G^x\) is given. Let \(\lambda_G\) be the Haar
system that~\(\lambda\) induces on~\(G\) in
Theorem~\ref{thm:Haar-sys-general}. We shall describe formula for
integration by~\(\lambda_G\) in terms of~\(\lambda\).

Following is the idea for finding this map: consider the commuting
Figure~\ref{fig:haar-measure-formula}.

\begin{figure}[htb]
  \centering
  \begin{tikzcd}
    G^x \ltimes G \arrow [rr, "k"] \arrow[dr, "p"]& & G^x \times G^x
    \arrow[dl, "q"]\\
    & G &
  \end{tikzcd}
  \caption{} \label{fig:haar-measure-formula}
\end{figure}

\noindent The maps in Figure~\ref{fig:haar-measure-formula} are as
follows: the transversal \(G^x\) is an equivalence between \(G^x_x\)
and \(G\); we consider the fibre product
\(G^x\times_{s|_{G^x}, \base[G], r} G\) as the transformation groupoid
\(G^x\rtimes G\). In Figure~\ref{fig:haar-measure-formula}, \(k\) is
the map of kinetics for the right multiplication action of~\(G\) on
\(G^x\). Thus
\[
  k\colon G^x\rtimes G \to G^x \times G^x, \quad k\colon (y,\gamma)
  \mapsto (y,y\gamma),
\]
where \((y,\gamma)\in G^x\rtimes G\). As \(G^x\) is an
\(G^x_x\)-\(G\)\nb-equivalence, this kinetics is an isomorphism of
topological groupoids when \(G^x\times G^x\) is considered as the
groupoid of the trivial equivalence relation. The inverse of~\(k\) is
given by
\[
  k\inverse(y,z) = (y,y\inverse z) \quad \text{ for } \quad (y,z)\in
  G^x\times G^x.
\]

In Figure~\ref{fig:haar-measure-formula}, the
map~\(p \colon G^x\rtimes G \to G\) is the projection onto the second
factor. And \(q\colon G^x \times G^x \to G\) is defined as
\(q(y,z) = y\inverse z\) for \((y,z) \in G^x\times G^x\). In
fact,~\(q\) is a quotient map: recall
Lemma~\ref{lem:equi-gpd-construction}. Applying this lemma to the
\(G_x^x\)-\(G\)\nb-equivalence \(G^x\), we see that
\(G_x^x\7 (G^x\times G^x) \iso G\) are groupoids. It can be checked
that~\(q\) is the quotient
map~\(G^x\times G^x\to G\iso G_x^x\7 (G^x\times G^x)\). One needs to
show that images of~\((y,z), (y',z')\in G^x\times G^x\) under~\(q\)
are same \emph{iff} \((y',z') = (\xi y, \xi z)\) for
some~\(\xi\in G_x^x\); we leave this as an easy exercise.

Finally, it is a direct computation that
Figure~\ref{fig:haar-measure-formula} commutes.

Our idea of getting the Haar system on~\(G\) is as follows: the given
measure~\(\lambda\) induces the Haar system
\(\lambda_2\ \defeq \{\delta_y \times \lambda\}_{y\in G^x}\) on the
groupoid~\(G^x \times G^x\) (cf. Example~\ref{exa:haar-sys}).  The
groupoid isomorphism \(k\), pulls back this Haar system to the Haar
system
\(k^*\lambda_2\defeq \{k^*(\delta_y \times \lambda)\}_{y\in G^x}\)
on~\(G^x\rtimes G\). Now, the point is that since \(G^x\) is the space
of units of the transformation groupoid~\(G^x\rtimes G\), the measures
in~\(k^*\lambda_2\) \emph{heuristically} reside on fibres
of~\(G\). Using the projection map~\(p\), we indent to figure out
these measures. And as it turns out, this plan works!

Let \(g\in \Contc(G)\) and \(w\in \base[G]\); our aim to describe
\(\int_G g\, \dd \lambda_G^w\). Firstly, we want to consider~\(g\) as
a function on the transformation groupoid~\(G^x\rtimes G\). For this,
we \emph{choose} \(\eta\in G_w^x\); such~\(\eta\) exists as~\(G\) is
transitive. Let \(\chi_{\{\eta\}}\) denote the characteristic function
of the set~\(\{\eta\}\). Then we consider the
function~\(\chi_{\{\eta\}}\times g\) on~\(G^x\rtimes G\). Clearly, the
restricted function \((\chi_{\{\eta\}}\times g)|_{r\inverse(\eta)}\)
is in~\(\Contc((G^x\rtimes G)^{\eta})\).  And, finally, we define
\[
  \int_G g \, \dd\lambda_G^w \defeq \int_G (\chi_{\{\eta\}}\times g)
  \,\dd k^*\lambda_2^y.
\]
Our next task is to simply the last formula.

By definition of the pullback measures, we can see that above formula
is essentially
\begin{multline*}
  \int_G g \, \dd\lambda_G^w = \int_G (\chi_{\{\eta\}}\times g)\circ
  k\inverse(y,z) \,\dd \lambda_2^y(y,z) \\= \int_G
  (\chi_{\{\eta\}}\times g)(y,y\inverse z) \,\dd \lambda(z) = \int_G
  g(\eta\inverse z) \,\dd \lambda(z).
\end{multline*}

\noindent To summarise, we \emph{claim} that for \(g\in \Contc(G)\)
and \(w\in \base[G]\),
\begin{equation}
  \label{eq:haar-measure-for-1}
  \int_G g \, \dd\lambda_G^w=  \int_{G^x}  g(\eta\inverse z)  \,\dd
  \lambda(z)
\end{equation}
where \(\eta\in G_w^x\) and the Haar system
\(\{\lambda_G^w\}_{w\in \base[G]}\) on~\(G\) induced by the
measure~\(\lambda\) on~\(G^x\). First of all, we need to check that
Equation~\eqref{eq:haar-measure-for-1} is independent of choice of
\(\eta\in G_w^x\). Choose another \(\xi\in G_w^x\). Then
\[
  \int_G g(\xi\inverse y) \,\dd \lambda(y) = \int_G
  g(\eta\inverse((\eta\xi\inverse) y)) \,\dd \lambda(y)
\]
where \(\eta\xi\inverse \in G_x^x\) as \(\eta,\xi\in G_w^x\). As
\(\lambda\) is \(G_x^x\)\nb-invariant, the last integral above equals
\[
  \int_G g(\eta\inverse y) \,\dd \lambda(y)
\]
proving that Equation~\eqref{eq:haar-measure-for-1} is well-defined.

\begin{remark}\label{rem:for-cont-of-lambda-G}
  Once again, recall that \(G^x\) is a
  \(G^x_x\)-\(G\)\nb-equivalence. And notice that, as~\(G\) is
  transitive, the quotient \(G^x_x\7 G^x\) for the left multiplication
  action of~\(G^x_x\) on the transversal~\(G^x\)
  is~\(\base[G]\). Therefore, if we quotient the
  \(G^x_x\)\nb-equivariant range map
  \(r \colon G^x\times G^x \to G^x\) by the left \(G^x_x\)\nb-action,
  then we get the range map \(r: G\to \base[G]\),
  cf.~Figure~\ref{fig:cont-of-measure}. In
  Figure~\ref{fig:cont-of-measure}, the vertical maps are the quotient
  maps.
  \begin{figure}[hbt]
    \centering
    \begin{tikzcd}
      G^x \times G^x
      \arrow[r, "r"'] \arrow[d, "q"]& G^x \arrow[d] \\
      G \arrow[r, "r"'] & \base[G]
    \end{tikzcd}
    \caption{} \label{fig:cont-of-measure}
  \end{figure}
  The range map of~\(G^x\times G^x\) carried the Haar
  system~\(\lambda_2\). One can notice that by construction the family
  of measures~\(\lambda_G\) along the range map of~\(G\) is the one
  derived from~\(\lambda_2\) from above quotient as in Lemma~1.3 (i)
  of~\cite{Renault1985Representations-of-crossed-product-of-gpd-Cst-Alg};
  the only \emph{twist} is that we used the groupoid~\(G^x\rtimes G\),
  that is isomorphic to the groupoid of trivial
  equivalence~\(G^x\times G^x\), for
  defining~\(\lambda_G\). Therefore, by virtue
  of~\cite[Lemma~1.3(i)]{Renault1985Representations-of-crossed-product-of-gpd-Cst-Alg},
  \(\lambda_G\) is a continuous family of measures.
\end{remark}

To formalise the idea behind Equation~\eqref{eq:haar-measure-for-1},
let \(x,w\in \base\) as before. Consider the left
translation~\(\LTr_\eta g\) for a function~\(g\in \Contc(G)\) by
\(\eta\in G_w^x\). That is, \(\LTr_\eta g \in \Contc(G^x)\) is defined
by \(\LTr_\eta g (\gamma) = g(\eta\inverse \gamma)\) for
\(\gamma\in G^x\). Choose \(C\subseteq G\) formed by collecting
\emph{exactly} one arrow~\(\eta\in G^x_v\) for the fixed
\(x\in \base[G]\) and ~\(v\) varying over~\(\base[G]\). Define the
translated measure~\(\LTr_{\eta\inverse}\lambda\) on~\(G^w\) by
\[
  \int_{G^w} f(\xi) \,\dd(\LTr_{\eta\inverse}\lambda)(\xi) \defeq
  \int_{G^x} \LTr_\eta f(\gamma) \, \dd\lambda(\gamma)
\]
where \(f\in \Contc(G)\).  Then
\(\lambda_G = \{\LTr_{\eta\inverse}\lambda\}_{\eta\in C}\).

\begin{proposition}
  \label{prop:formula-1}
  Let \(G\) be a locally compact locally, trivial groupoid. Let
  \(x\in \base[G]\). Assume that \(\lambda\) is a
  \(G_x^x\)\nb-invariant measure on the transversal~\(G^x\) for the
  left multiplication of~\(G_x^x\). Then
  \(\lambda_G \defeq \{\lambda_G^w\}_{w\in \base[G]}\) defined in
  Equation~\eqref{eq:haar-measure-for-1} is a Haar system on~\(G\).
\end{proposition}
\begin{proof}
  \emph{Support condition:} Let \(\eta\in G_w^x\) as before. Then we
  note that the multiplication by~\(\eta\inverse\) is a
  homeomorphism~\(G^x \to G^w\) and~\(\LTr_{\eta\inverse}\lambda\) is
  the push forward of the measure~\(\lambda\) along this
  homeomorphism. As~\(\lambda\) has full support,
  \(\LTr_{\eta\inverse}\lambda\) has full support.

  \noindent \emph{Translation invariance:} This is a straightforward
  computation: let \(g\in \Contc(G)\) and \(\xi\in G\).  Then
  \[
    \int_G g(\xi \gamma) \,\dd\lambda_G^{s(\xi)}(\gamma) = \int_G
    \LTr_{\xi\inverse} g(\gamma) \,\dd\lambda_G^{s(\xi)}(\gamma)
    \defeq \int_{G^x} \LTr_{\eta}\LTr_{\xi\inverse} g(\gamma)
    \,\dd\lambda(\gamma)
  \]
  where \(\eta\in G_{s(\xi)}^x\). But
  \(\LTr_{\eta}\LTr_{\xi\inverse} = \LTr_{\eta\xi\inverse}\) and
  \(\eta\xi\inverse \in G_{r(\xi)}^x\). Therefore, the last term above
  equals
  \[
    \int_{G^x} \LTr_{\eta \xi\inverse} g(\gamma) \,\dd\lambda(\gamma)
    \defeq \int_{G} g(\gamma) \,\dd\lambda_G^{r(\xi)}(\gamma)
  \]
  which proves the invariance.

  Finally, the continuity of~\(\lambda_G\) is shown in
  Remark~\ref{rem:for-cont-of-lambda-G}.
\end{proof}

\subsubsection{Haar system in terms of measure on the space of units:}

Let \(G\) be a locally compact, locally trivial groupoid and let
\(\nu\) be a fully supported measure on~\(\base[G]\). We provide a
formula for the Haar system on~\(G\)
using~\(\nu\). Fix~\(x\in\base[G]\). Recall from the proof
of~\((3)\implies (2)\) in Theorem~\ref{thm:Haar-sys-general} that the
\(G_x^x\)\nb-invariant measure~\(\lambda\) that~\(\nu\) induces
on~\(G^x\) is~\(\nu\circ \mu\) where~\(\mu\) is the continuous family
of measures along the local
homeomorphism~\(s|_{G^x}\colon G^x \to \base[G]\); \(\mu\)~consisting
of atomic measures. Thus, for \(g\in \Contc(G^x)\),

\begin{equation}
  \label{eq:measure-lambda}
  \int_{G^x} g(\xi) \, \dd\lambda(\xi) \defeq \int_{\base[G]} \Big(\sum_{\gamma\in
    s|_{G^x}\inverse(w)} g(\gamma) \Big) \, \dd\nu(w).
\end{equation}

Thus, if \(f\in \Contc(G)\) and \(v\in \base[G]\), then, using
Equation~\eqref{eq:haar-measure-for-1}, the Haar
system~\(\lambda_G \defeq \{\lambda_G^w\}_{w\in \base[G]}\) on~\(G\)
is given by

\begin{equation}\label{eq:nu-to-haar-system}
  \int_G f(\xi) \, \dd\lambda_G^w(\xi) = \int_{G^x} \LTr_\eta f(\xi) \,
  \dd\lambda(\xi) = \int_{\base[G]} \Big(\sum_{\gamma\in
    s|_{G^x}\inverse(w)} f(\eta\inverse \gamma) \Big) \,
  \dd\nu(w)
\end{equation}
where \(\eta\in G_w^x\).

In particular, if \(G\) is the fundamental groupoid~\(\FGd(X)\) as in
Corollary~\ref{cor:Haar-sys-general} then
Equations~\eqref{eq:measure-lambda} and~\eqref{eq:nu-to-haar-system}
are respectively given as follows:
\begin{align*}
  \int_{\tilde{X}^x} g(\xi) \, \dd\lambda(\xi) &= \int_{\tilde{X}^x} \Big(\sum_{\gamma\in
                                                 {\phi}\inverse(w)}
                                                 g(\gamma) \Big) \,
                                                 \dd\nu(w)\\
  \int_{\FGd(X)} f(\xi) \, \dd\lambda_{\FGd(X)}^w(\xi) &= \int_{X} \Big(\sum_{\gamma\in
                                                         \phi\inverse(w)} f(\eta\inverse \gamma) \Big) \,
                                                         \dd\nu(w)
\end{align*}
where \(\tilde{X}^x = r\inverse(x)\) is the simply connected covering
space associated with~\(x\in X\), \(\eta \in \FGd(X)^x_w\), and
\(\phi\colon \tilde{X} \to X\) is the covering map.

\begin{proposition}\label{prop:cst-of-fgd}
  Let \(X\) be a path connected, locally path connected, semilocally
  simply connected, locally compact, Hausdorff and second countable
  space. Then~\(\FGd(X)\) carries a Haar system and
  \(\Cst(\FGd(X)) \iso \Cst(\FGp(X,x))\otimes \mathbb{K}\)
  where~\(\mathbb{K}\) is the \(\Cst\)\nb-algebra of compact operators
  on a separable Hilbert space.
\end{proposition}
\begin{proof}
  Corollary~\ref{cor:sec-cble-gpd-Haar-sym} says that~\(\FGd(X)\) has
  a Haar system. And the isomorphism of \(\Cst\)\nb-algebras follow
  directly
  from~\cite[Theorem~3.1]{Muhly-Renault-Williams1987Gpd-equivalence}
\end{proof}

\begin{proposition}\label{prop:cst-of-fgd-of-gp}
  Let \(X\) be a locally compact, path connected, locally path
  connected and semilocally simply connected group. Let
  \(p\colon H\to X\) be its universal cover. Then \(\Cst(\FGd(X))\) is
  isomorphic to the crossed product \(H\ltimes_p \Contz(X)\).
\end{proposition}
\begin{proof}
  Follows from Example~\ref{exa:fgd-of-gp} and
  Theorem~\ref{thm:Haar-measure-Gpd-of-Gp}.
\end{proof}

Proposition~\ref{prop:cst-of-fgd} and
Proposition~\ref{prop:cst-of-fgd-of-gp} together imply the following:
Assume \(p\colon H\to X\) as in
Proposition~\ref{prop:cst-of-fgd-of-gp}, plus that~\(X\) is second
countable. Let \(K\) be the kernel of~\(p\). Then \(K \iso \FGp(X,e)\)
where \(e\in X\) is the unit, see the remark after~\cite[Theorem 2.21]{Holkar-Hossain2023Top-FGd-I}. Then the crossed product
\(H\ltimes_p \Contz(X) \iso \Cst(K)\otimes \mathbb{K}\).

\medskip

\paragraph{\itshape Acknowledgement:} We are thank to Dheeraj Kulkarni
for many fruitful discussions. We thank SERB and CSIR for their
grants---SERB's SRG/2020/001823 grant to the first author and CSIR's
09/1020(0159)/2019-EMR-I grant to the second author---which supported
this work.

\end{document}